\documentclass[11pt, reqno]{amsart}
\usepackage[margin=1in]{geometry}
\usepackage{tikz-cd}

\usepackage{amssymb,amsfonts,amsmath}
\usepackage[all,arc]{xy}
\usepackage{enumerate}
\usepackage{mathrsfs}
\usepackage{mathtools}
\usepackage{thmtools}
\usepackage{enumitem}
\usepackage{bm}
\usepackage{graphicx}
\usepackage{hyperref}
\hypersetup{
	colorlinks=true,
	linkcolor=blue,
	filecolor=magenta,      
	urlcolor=cyan,
	citecolor=blue,
	pdfpagemode=FullScreen,
}
\usepackage{url}
\usepackage{xcolor}
\PassOptionsToPackage{cmyk}{xcolor}

\newtheorem{thm}{Theorem}[section]

\newtheorem{prop}[thm]{Proposition}
\newtheorem*{prop*}{Proposition}
\newtheorem{lem}[thm]{Lemma}

\newtheorem*{thm*}{Theorem}
\newtheorem*{lem*}{Lemma}

\theoremstyle{definition}

\newtheorem{exmp}[thm]{Example}

\newtheorem{notn}[thm]{Notation}

\newtheorem*{defn*}{Definition}
\newtheorem*{fact*}{Fact}

\theoremstyle{remark}
\newtheorem{rem}[thm]{Remark}

\newtheorem*{rem*}{Remark}
\newtheorem*{rems*}{Remarks}

\declaretheoremstyle[notefont=\bfseries,notebraces={}{},%
headpunct={},postheadspace=1em]{mystyle}

\newcommand{\Q}{\mathbf{Q}}
\newcommand{\Z}{\mathbf{Z}}
\newcommand{\R}{\mathbf{R}}
\newcommand{\C}{\mathbf{C}}

\newcommand{\RP}{\R\mathrm{P}}

\DeclareMathOperator{\Isom}{Isom}

\DeclareMathOperator{\tr}{tr}

\newcommand{\SL}[2]{\mathrm{SL}({#1},{#2})}
\newcommand{\PSL}[2]{\mathrm{PSL}({#1}, {#2})}

\usepackage[backend=biber,style=alphabetic, citestyle=alphabetic, doi=false, isbn=false, giveninits=true, url=false, maxbibnames=6, maxcitenames=6, maxalphanames=6]{biblatex}
\renewbibmacro{in:}{}
\title{Freely $2$-periodic knots have two canonical components}
\author{Keegan Boyle and Nicholas Rouse}

\begin{document}
\begin{abstract}
We prove that the $\SL{2}{\C}$ character variety of a hyperbolic, freely 2-periodic knot has two canonical components. We also prove that the hyperbolic torsion polynomial of such a knot satisfies a factorization condition which seems to be particularly effective at identifying freely 2-periodic knots.
\end{abstract}
	\maketitle
	\section{Introduction}
	An important invariant of a hyperbolic knot $K$ is the $\SL{2}{\C}$ character variety $X(K)$, introduced in \cite{MR0683804}. The character variety has been used to detect the existence of essential surfaces \cite{MR0683804} and in the proof of the cyclic surgery theorem \cite{CGLS} among many other applications such as \cite{CCGLS, Dunfield99}. A knot with a complete hyperbolic metric on the complement naturally comes with a discrete faithful representation to $\PSL{2}{\C} \cong \Isom^+(\mathbf{H}^3)$, which has exactly two lifts to $\SL{2}{\C}$ (see e.g., \cite[\S 4.1]{DunfieldFriedlJackson}). Any component of $X(K)$ that contains one of these two discrete faithful representations is called a \emph{canonical component} (or sometimes ``excellent component'' or ``Dehn surgery component''), and is of particular interest in the study of $X(K)$; see for example  \cite{KlaffTillmanBirat, ChuSurfaces, RouseSevenFour, CRS}.
	
	Certain classes of knots, such as hyperbolic $(-2,3,n)$-pretzel knots \cite{MR1949779} and $2$-bridge knots (see \cite{MR2484712} and \cite[Remark 3.4]{MR3922034}) are known to have a single canonical component, but it seems to be unknown whether a hyperbolic knot can have more than one canonical component \cite[Remark 3.3]{MR3922034}. We show that there are infinitely many knots with two canonical components by studying knots with a particular type of symmetry. Specifically, we consider knots that are invariant under the free order-$2$ symmetry on $S^3$, called \emph{freely 2-periodic} knots (see e.g., \cite[\S 1.1]{boyle2023classification}). For an example see Figure \ref{fig:10_157}. Character varieties of symmetric knots have been studied in \cite{PPInvariantCharVar} and \cite{HLMa} though neither paper considers the case of free $2$-periods. 
	
	Given a freely 2-periodic knot, there is a corresponding quotient knot $\overline{K}$ in $\RP^3$. The character variety is functorial, and so there is a map of character varieties $X(\overline{K}) \to X(K)$.  By studying this map, we prove the following theorem.

	\begin{thm} \label{thm:main}
		If $K \subseteq S^3$ is a hyperbolic, freely $2$-periodic knot, then the $\SL{2}{\C}$ character variety of $\pi_1(S^3 \setminus K)$ has two canonical components.
	\end{thm}
	As a corollary this gives an obstruction to free $2$-periodicity. For example we obtain a new proof that $2$-bridge knots and $(-2,3,n)$-pretzel knots are not freely $2$-periodic.
	
	In general canonical components of character varieties are difficult to compute. Despite some effort, we were unable to compute the irreducible components of any freely $2$-periodic knot explicitly. 
\begin{figure}
	\scalebox{.3}{\includegraphics{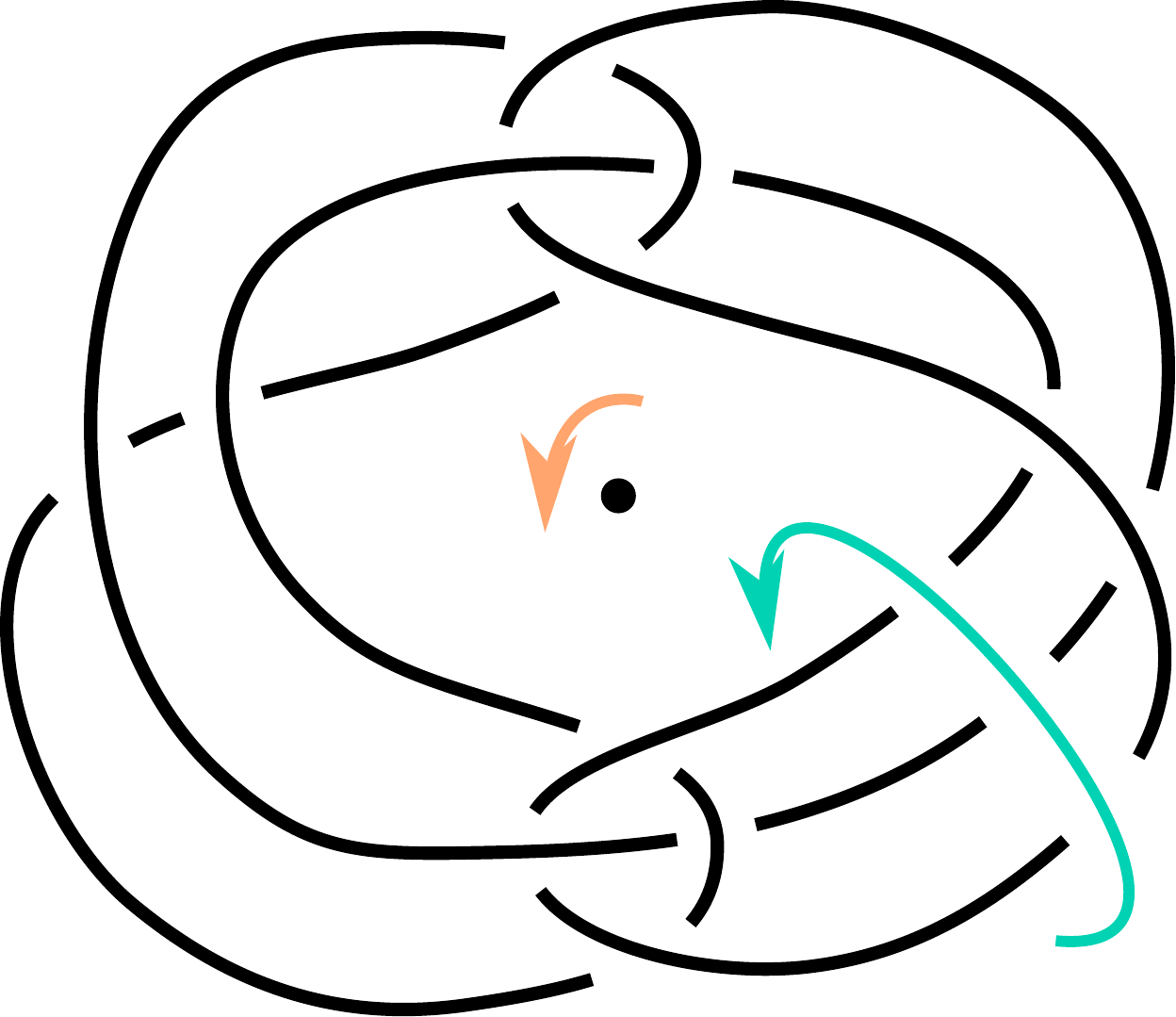}}
	\caption{The freely 2-periodic knot $10_{157}$. The symmetry is the composition of $\pi$ rotations around two disjoint circles, indicated by the two arrows. By Theorem \ref{thm:main} $X(10_{157})$ has two canonical components.}
	\label{fig:10_157}
\end{figure}
	\begin{rem}
	In our definition of a canonical component, we insist that the discrete faithful representations respect the orientation on the knot exterior. However, the complex conjugate of a discrete faithful representation of the knot group is also discrete and faithful, corresponding to the manifold with reversed orientation. See \cite[\S 4]{MR3652806} for an example of a $1$-cusped hyperbolic manifold for which the character of a discrete faithful representation and its complex conjugate lie on different components of the character variety. In contrast, Theorem \ref{thm:main} guarantees two distinct components containing characters of discrete faithful representations corresponding to the same orientation.
	\end{rem}

	\begin{rem}
	While the two canonical components are abstractly isomorphic (see for example, \cite[Remark 3.3]{MR3922034}), our proof furnishes a geometric distinction. In particular, let $X_+$ and $X_-$ be the components containing characters of discrete faithful representations that send a meridian to $+2$ and $-2$ respectively. Then $X_-$ is in the image of the map on character varieties induced by the $2$-fold cover $\widehat{p} \colon X(\overline{K}) \to X(K)$ while $X_+$ is not. For any hyperbolic knot with two canonical components, there is also a distinction coming from the characters of Dehn surgeries. Let $S^3_{p/q}(K)$ be $p/q$ Dehn surgery on $K$. If $|H_1(S^3_{p/q}(K);\Z)|$ is odd, then up to complex conjugation there is a unique character corresponding to a discrete faithful representation $\pi_1(S^3_{p/q}(K)) \to \SL{2}{\C}$, which can lie on at most one canonical component of $X(K)$. That is, the canonical components are distinguished by which Dehn surgery characters lie on each component.
	\end{rem}

	While proving Theorem \ref{thm:main}, we see that exactly one of the two discrete faithful $\SL{2}{\C}$ representations for a freely 2-periodic knot is the restriction of a discrete faithful representation for the quotient knot. This implies, via \cite[Theorem 5]{HillmanLivingstonNaik}, a particular factorization condition on the hyperbolic torsion polynomial, answering Question 3 in \S 1.17 of \cite{DunfieldFriedlJackson}.
\begin{thm} \label{thm:hyperbolictorsion}
	Let $K$ be a freely $2$-periodic hyperbolic knot with quotient knot $\overline{K}$. Let $\mathcal{T}_K$, $\mathcal{T}_{\overline{K}}$ be the hyperbolic torsion polynomials for $K$ and $\overline{K}$ respectively. Then
	\[
	\mathcal{T}_K(-t^2) = \mathcal{T}_{\overline{K}}(t)\mathcal{T}_{\overline{K}}(-t),
	\]
	where this factorization occurs over the trace field of $S^3 \setminus K$.
\end{thm}
In Example \ref{ex:12n553} we use Theorem \ref{thm:hyperbolictorsion} to show that the knots $12n553$ and $12n556$ are not freely 2-periodic, despite passing all obstructions to free periodicity that we could find in the literature.
\begin{rem}
In Theorem \ref{thm:hyperbolictorsion}, $\mathcal{T}_K(t)$ is the twisted Alexander polynomial corresponding to the discrete faithful representation $\pi_1(S^3 \setminus K) \to \SL{2}{\C}$ in which the trace of a meridian is $2$. In all of our computations, the twisted Alexander polynomial corresponding to the other discrete faithful representation does not factor as in Theorem \ref{thm:hyperbolictorsion}.
\end{rem}
\subsection{Questions} Based on some computational work, we find it natural to ask the following.
	\begin{enumerate}
		\item Is there a hyperbolic knot in $S^3$ that is not freely $2$-periodic but does have $2$ canonical components in its $\SL{2}{\C}$ character variety?
		\item Is there a hyperbolic knot in $S^3$ that is not freely $2$-periodic whose hyperbolic torsion polynomial admits a factorization of the form in Theorem \ref{thm:hyperbolictorsion}?
	\end{enumerate}
\subsection{Acknowledgments} We thank Ben Williams for helpful conversations. We thank Liam Watson and Ben Williams for their feedback on a draft of the paper.

	\section{Proof of Theorem \ref{thm:main}} \label{sec:main}
	\begin{notn}
	In this section, $K$ refers to hyperbolic, freely $2$-periodic knot in $S^3$, and $\overline{K}$ refers to its quotient knot in $\RP^3$. We slightly abuse the notation $S^3 \setminus K$ to refer to the complement of a tubular neighborhood of $K$, so that $\partial(S^3 \setminus K)$ is a torus. We use $\RP^3 \setminus \overline{K}$ similarly.
	\end{notn}

	\begin{lem} \label{lem:rp3homology}
		Let $K$ be a freely $2$-periodic knot in $S^3$ and $\overline{K} \subseteq \RP^3$ its quotient knot. Then $\overline{K}$ represents the nontrivial homology class of $\RP^3$, $\mathrm{H}_1(\RP^3 \setminus \overline{K}; \Z) = \Z$, and the map $\mathrm{H}_1(\partial(\RP^3 \setminus \overline{K}); \Z) \to \mathrm{H}_1(\RP^3 \setminus \overline{K}; \Z)$ is surjective.
	\end{lem}
	\begin{proof}
		Note that the quotient knot $\overline{K}$ represents the nontrivial homology class because knots representing the trivial homology class in $\RP^3$ lift to $2$-component links in $S^3$. It follows that the homology of the complement of the quotient knot is $\Z$. See, e.g., \cite[Corollary 5]{ManfrediHomology}. The proof of the same corollary in fact shows that the homology of $\RP^3 \setminus \overline{K}$ is generated by a boundary element. 
	\end{proof}
	
	\begin{lem} \label{lem:pi_1_ker}
		Let $K$ be a freely $2$-periodic knot in $S^3$ and $\overline{K} \subseteq \RP^3$ its quotient knot. Let $\epsilon \colon \pi_1(\RP^3 \setminus \overline{K}) \rightarrow \Z/2\Z$ be the non-trivial map. Then $\pi_1(S^3 \setminus K) = \ker{\epsilon}$.
	\end{lem}
	\begin{proof}
		Note that $\mathrm{H}_1(\RP^3 \setminus \overline{K}; \Z) = \Z$ by Lemma \ref{lem:rp3homology}. Then there is a unique nontrivial map $\epsilon \colon \pi_1(\RP^3 \setminus \overline{K}) \to \Z/2\Z$. Therefore $S^3 \setminus K$ is the unique $2$-fold connected cover, and $\ker{\epsilon} = \pi_1(S^3 \setminus K)$.
	\end{proof}
	\begin{lem} \label{lem:homologyMap}
		There are bases $\{m, \ell\}$ and $\{\alpha, \overline{\ell}\}$ for $\mathrm{H}_1(\partial(S^3 \setminus K); \Z)$ and $\mathrm{H}_1(\partial(\RP^3 \setminus \overline{K}); \Z)$, respectively such that
		\begin{enumerate}[label={(\arabic*)},ref={(\arabic*)}]
			\item \label{homologyMap:merLong} $m$ and $\ell$ are represented by a meridian and longitude respectively for $K$,
			\item \label{homologyMap:alpha} $\alpha$ generates $\mathrm{H}_1(\RP^3 \setminus \overline{K}; \Z)$,
			\item \label{homologyMap:boundsSurface} $\overline{\ell}$ is represented by the boundary of an embedded, oriented surface in $\RP^3 \setminus \overline{K}$,
			\item \label{homologyMap:matrix} the map $f\colon \mathrm{H}_1(\partial(S^3 \setminus K); \Z) \to \mathrm{H}_1(\partial(\RP^3 \setminus \overline{K}); \Z)$  induced by the $2$-fold cover is given by
			\[
			\begin{pmatrix}
				2 & 0 \\
				1 & 1
			\end{pmatrix}
			\]
			with respect to the ordered bases $\{m,\ell\}$ and $\{\alpha, \overline{\ell}\}$.
		\end{enumerate} 
	\end{lem}
\begin{proof}
	Let $\overline{\ell}$ be the class of the homological longitude of $\overline{K}$ so that \ref{homologyMap:boundsSurface} is automatic. Indeed, using the identification $\mathrm{H}^1(\RP^3 \setminus \overline{K};\Z) \cong [\RP^3 \setminus \overline{K}, S^1]$ we can see that $\overline{\ell}$ is the boundary of an oriented surface $\overline{\Sigma}$ in $\RP^3$ which is the preimage of a regular point of a map $\RP^3 \setminus \overline{K} \to S^1$ generating $\mathrm{H}^1(\RP^3 \setminus \overline{K}; \Z)$. We can then produce a pair of Seifert surfaces for $K$ by lifting $\overline{\Sigma}$ from $\RP^3$ to $S^3$. The boundary of one of these surfaces is a homological longitude $\ell$ in $\mathrm{H}_1(\partial(S^3 \setminus K);\Z)$ whose image $f(\ell)$ in $\mathrm{H}_1(\partial(\RP^3 \setminus \overline{K}); \Z)$ is the homological longitude $\overline{\ell}$. We next need a second generator $\alpha$ for $\mathrm{H}_1(\partial(\RP^3 \setminus \overline{K}); \Z)$, and to verify items \ref{homologyMap:alpha} and \ref{homologyMap:matrix}. Since the map $\mathrm{H}_1(\partial(\RP^3 \setminus \overline{K}); \Z) \to \mathrm{H}_1(\RP^3 \setminus \overline{K};\Z)$ induced by inclusion is surjective by Lemma \ref{lem:rp3homology}, any choice of second generator for $\mathrm{H}_1(\partial(\RP^3 \setminus \overline{K}); \Z)$ must generate $\mathrm{H}_1(\RP^3 \setminus \overline{K}; \Z)$, satisfying \ref{homologyMap:alpha}. Additionally, the determinant of the map $f \colon \mathrm{H}_1(\partial(S^3 \setminus K); \Z) \to \mathrm{H}_1(\partial(\RP^3 \setminus \overline{K}); \Z)$ is $\pm 2$ by Lemmas \ref{lem:rp3homology} and \ref{lem:pi_1_ker}, noting that abelianization is right-exact. We may arrange the determinant to be $+2$ by a judicious choice of orientation on a meridian $m$ of $K$. Hence for any choice of $\alpha$, the map $f\colon \Z\langle m, \ell \rangle \to \Z\langle \alpha,\overline{\ell} \rangle$ is represented by the matrix
		\[
			\begin{pmatrix}
				2 & 0 \\
				x & 1
			\end{pmatrix},
		\]
	where $x$ depends on the choice of $\alpha$. 

	Next, observe that $f(m)$ is represented by a simple closed curve in $\partial(\RP^3 \setminus \overline{K})$. Indeed, $m$ is represented by a simple closed curve (which we also call $m$) in $\partial(S^3 \setminus K)$, and the non-trivial deck transformation $\tau$ is the restriction of the freely 2-periodic symmetry on $S^3$ which acts freely on $K$ so that $\tau(m)$ is disjoint from $m$. Hence $f(m)$ is also a simple closed curve. Furthermore, a simple closed curve represents either a primitive or trivial class in the homology of $S^1 \times S^1 \cong \partial(\RP^3 \setminus \overline{K})$ (see e.g., \cite{MR0425967} or \cite{MR0454995}). We conclude that $x$ must be odd, or $f(m)$ would represent an even (and hence non-primitive) class in $\mathrm{H}_1(\partial(\RP^3 \setminus \overline{K});\Z)$. Now for any $n \in \Z$ we can change basis by replacing $\alpha$ with $n\overline{\ell} + \alpha$, which represents the same homology class in $\RP^3 \setminus \overline{K}$. This change of basis changes $x$ by $2n$ so that for an appropriate choice of $n$ we have $x = 1$. 
\end{proof}

\begin{prop} \label{prop:traceMinusTwo}
	Let $\rho_{\overline{K}} \colon \pi_1(\RP^3 \setminus \overline{K}) \to \SL{2}{\C}$ be a discrete faithful representation, $\rho_K = \rho_{\overline{K}}\big|_{\pi_1(S^3 \setminus K)}$, and $m \in \pi_1(S^3 \setminus K)$ be a meridian. Then $\rho_K$ is discrete and faithful, and $\tr{\rho_K(m)} = -2$.
\end{prop}

\begin{proof}
	Throughout we freely identify elements of $\pi_1(\partial(S^3 \setminus K))$ and $\pi_1(\partial(\RP^3 \setminus \overline{K}))$ with $\mathrm{H}_1(\partial(S^3 \setminus K); \Z)$ and $\mathrm{H}_1(\partial(\RP^3 \setminus \overline{K}); \Z)$ respectively. The restriction of a discrete, faithful representation is discrete and faithful. Meridians of knots in $S^3$ map to parabolic elements (see e.g., \cite[Exercise 5.43]{PurcellHypKnot}) under a discrete faithful representation, so $\tr{\rho_K(m)} = \pm 2$. Suppose that $\tr{\rho_K(m)}$ is $+2$. By Lemma \ref{lem:homologyMap}, the meridian $m$ maps to $\alpha^2 \overline{\ell}$ in the notation of that lemma, which we implicitly consider as elements of $\pi_1(\RP^3 \setminus \overline{K})$. Further, \cite[Corollary 2.4]{DCalegariTrace} implies that $\tr{\rho_{\overline{K}}(\overline{l})} = -2$. Then using $\SL{2}{\C}$ trace relations (see e.g., \cite[\S 3.4]{MR}), we have
		\[
		\begin{aligned}
			2=\tr{\rho_K(m)} = \tr{\rho_{\overline{K}}(\alpha^2\overline{\ell})} &= \tr{\rho_{\overline{K}}(\alpha)}\tr{\rho_{\overline{K}}(\alpha \overline{\ell})} - \tr{\rho_{\overline{K}}(\alpha \overline{\ell}^{-1} \alpha^{-1})}\\
			&= \tr{\rho_{\overline{K}}(\alpha)}\tr{\rho_{\overline{K}}(\alpha \overline{\ell})} - \tr{\rho_{\overline{K}}(\overline{\ell})} \\
			&= \tr{\rho_{\overline{K}}(\alpha)}\tr{\rho_{\overline{K}}(\alpha \overline{\ell})} + 2.
		\end{aligned}
		\]
	Hence, 
		\[
		\tr{\rho_{\overline{K}}(\alpha)}\tr{\rho_{\overline{K}}(\alpha \overline{\ell})} = 0.
		\]
	However, traceless matrices in $\SL{2}{\C}$ are torsion, and $\pi_1(S^3 \setminus \overline{K})$ is torsion-free. This is a contradiction.
\end{proof}

The following lemma seems to be well-known in the literature, but we were unable to find a specific reference, so we include a proof.
\begin{lem} \label{lem:noGarbage}
	Let $\Gamma$ be a torsion-free group with a representation $\rho \colon \Gamma \to \SL{2}{\C}$. Let $\Gamma'$ be a finite index subgroup of $\Gamma$ such that $\rho|_{\Gamma'}$ is discrete and faithful. Then $\rho$ is discrete and faithful.
\end{lem}
\begin{proof}
	For discreteness, consider a Cauchy sequence $\gamma_i$ of elements of $\rho(\Gamma)$. Since $\Gamma'$ is of finite index in $\Gamma$, there is a coset $C$ of $\Gamma'$ such that there exists a subsequence $(\gamma_{i_j})$ of $(\gamma_i)$ contained in $\rho(C)$. Choose an element $c \in C$. Then $(\gamma_{i_j} c^{-1})$ is a sequence of elements of $\rho(\Gamma')$. This sequence is Cauchy since multiplication by $c^{-1}$ is continuous, so the sequence is eventually constant because $\rho(\Gamma')$ is discrete. Hence $(\gamma_{i_j})$ is eventually constant. This argument applies to each coset that has infinitely many elements of $(\gamma_i)$. As there are finitely many cosets, $(\gamma_i)$ is eventually constant.
	
	For faithfulness, let $n$ be the index of $\Gamma'$ in $\Gamma$. Take a nontrivial element $\gamma \in \Gamma$. We have $\gamma^n \in \Gamma'$, and $\gamma^n \neq \mathrm{Id}$ since $\Gamma$ is torsion-free. Since $\rho|_{\Gamma'}$ is faithful, we have $\rho(\gamma^n) \neq \mathrm{Id}$, and hence $\rho(\gamma) \neq \mathrm{Id}$.
\end{proof}
	\begin{proof}[Proof of Theorem \ref{thm:main}]
	Let $p \colon S^3 \setminus K \to \RP^3 \setminus \overline{K}$ be the $2$-fold cover. Let $X(K)$ and $X(\overline{K})$ be respectively the $\SL{2}{\C}$ character varieties of $\pi_1(S^3 \setminus K)$ and $\pi_1(\RP^3 \setminus \overline{K})$. Let $\widehat{p} \colon X(\overline{K}) \to X(K)$ be the map induced by $p$ (see e.g. \cite[\S 3]{BoyerLuftZhang} for details of this map). This map takes a character of a representation $\pi_1(\RP^3 \setminus \overline{K}) \to \SL{2}{\C}$ to the character of the representation $\pi_1(S^3 \setminus K)$ obtained by restriction. Let $X(\overline{K})_0$ be a canonical component of $X(\overline{K})$. By \cite[Proposition 3.1]{BoyerLuftZhang}, $\widehat{p}\big(X(\overline{K})_0\big)$ is a closed, irreducible algebraic curve. Moreover, the character of a representation on $X(\overline{K})_0$ maps to the character of the restriction of a representation to the index $2$ subgroup. In particular, the character of a discrete, faithful representation maps to the character of a discrete, faithful representation, and so $\widehat{p}\big(X(\overline{K})_0\big)$ is a canonical component, which we henceforth denote by $X(K)_0$. 
	
	Since $\widehat{p}|_{X(\overline{K})_0}$ surjects onto its image $X(K)_0$ which is a closed, irreducible curve, any point of $X(K)$ not in this image lies on a irreducible component distinct from $X(K)_0$. There are two lifts $\rho_+$ and $\rho_-$ of the discrete faithful representation $\pi(S^3 \setminus K) \rightarrow \PSL{2}{\C}$ to $\SL{2}{\C}$. If $m$ is a meridian of $K$, then $\tr(\rho_+(m)) = 2$ and $\tr(\rho_-(m)) = -2$. See, for example, \cite[\S 4.1]{DunfieldFriedlJackson}. By Proposition \ref{prop:traceMinusTwo}, the character of a discrete faithful representation in $X(\overline{K})$ maps to the character of a representation $\rho$ with $\tr{\rho(m)} = -2$. Moreover, by Lemma \ref{lem:noGarbage}, only characters of discrete faithful representations can map to characters of discrete faithful representations, so we conclude that the character of $\rho_+$ is not in the image of $\widehat{p}|_{X(\overline{K})_0}$.
\end{proof}
\begin{rem}
Theorem \ref{thm:main} can also be proved along the lines of \cite[Proposition 7.1]{BoyerLuftZhang}, although the method we use in this paper is more explicit.
\end{rem}

\section{Hyperbolic torsion for freely 2-periodic knots}
The hyperbolic torsion polynomial $\mathcal{T}_K(t)$ is the twisted Alexander polynomial associated to the discrete faithful representation of $\pi_1(S^3 \setminus K) \to \SL{2}{\C}$ that takes a meridian to a matrix with trace $+2$. The twisted Alexander polynomial associated to the other discrete faithful representation is $\mathcal{T}_K(-t)$ (\cite[Remark 4.4]{DunfieldFriedlJackson}).
\begin{proof}[Proof of Theorem \ref{thm:hyperbolictorsion}]
Let $\rho_- \colon \pi_1(S^3 \setminus K) \to \SL{2}{\C}$ be the discrete faithful representation that takes a meridian to a matrix with trace $-2$ so that $\mathcal{T}_K(-t)$ is the associated twisted Alexander polynomial. By \cite[Theorem 5]{HillmanLivingstonNaik}, it suffices to check that the discrete faithful representation $\rho_-$ is a restriction of a discrete, faithful representation of $\pi_1(\RP^3 \setminus \overline{K})$. This is the content of Proposition \ref{prop:traceMinusTwo}. 

It remains to show that $\mathcal{T}_K(-t^2)$ factors over the trace field of $S^3 \setminus K$. As a consequence of \cite[Theorem 4.2.1]{MR} and Lemma \ref{lem:rp3homology}, the invariant trace field of $\RP^3 \setminus \overline{K}$ is equal to its trace field. This field coincides with the (invariant) trace field of $S^3 \setminus K$ because the invariant trace field is a commensurability invariant (\cite[Theorem 3.3.4]{MR}).  It follows from \cite[Theorem 1.1(b)]{DunfieldFriedlJackson} that the coefficients of $\mathcal{T}_K(t)$ and $\mathcal{T}_{\overline{K}}(t)$ are in this common field. We conclude that the factorization appearing in \cite[Theorem 5]{HillmanLivingstonNaik} occurs over the trace field of $S^3 \setminus K$.
\end{proof}
We conclude with a pair of examples.
\begin{exmp}
	We consider the freely $2$-periodic knot $K = 10_{157}$ of the SnapPy census shown in Figure \ref{fig:10_157}. The trace field of $K$ is $\Q(z)$ where $z$ satisfies $z^3 + z - 1$; the discriminant of this number field is $-31$. The hyperbolic torsion of $K$, computed with \cite{SnapPyNT}, an extension of \cite{SnapPy} written by the second author, is
	\[
	\begin{aligned}
		\mathcal{T}_K(t) &= t^{10} + (-2 z^{2} - 3 z - 7) t^{9} + (12 z + 29) t^{8} + (5 z^{2} - 4 z - 72) t^{7} + (z^{2} - 11 z + 115) t^{6} \\
		&\quad + (-4 z^{2} + 16 z - 128) t^{5} + (z^{2} - 11 z + 115) t^{4} + (5 z^{2} - 4 z - 72) t^{3} + (12 z + 29) t^{2} \\
		&\quad +  (-2 z^{2} - 3 z - 7) t + 1.
	\end{aligned}	
	\]
	This polynomial is irreducible over $\Q(z)$ as is $\mathcal{T}_K(t^2)$. However $\mathcal{T}_K(-t^2)=f(t)f(-t)$, where
	\[
	\begin{aligned}
		f(t) &= t^{10} + (-z^{2} - z + 1) t^{9} + 5 t^{8} + (-3 z + 2) t^{7} + (z^{2} - 2 z + 7) t^{6} - 4 z t^{5} \\
		&\quad + (z^{2} - 2 z + 7) t^{4} + (-3 z + 2) t^{3} + 5 t^{2} + (-z^{2} - z + 1) t + 1.
	\end{aligned}
	\]
	That is, to apply \cite[Theorem 5]{HillmanLivingstonNaik} to the hyperbolic torsion polynomial, one must take the \textit{opposite} sign convention for the trace of a meridian of the one taken in \cite{DunfieldFriedlJackson}.
\end{exmp}
\begin{exmp} \label{ex:12n553}
Consider the knot $K = 12n553$. A computation with SnapPy \cite{SnapPy} shows that $K$ has an order-$2$ symmetry that preserves the orientation on $S^3$ and an orientation on $K$. Classical techniques such as the Alexander polynomial factorization conditions of Hartley \cite{Hartley} and Murasugi \cite{Murasugi} cannot determine whether this symmetry is 2-periodic or freely 2-periodic. However, Theorem \ref{thm:hyperbolictorsion} does. Indeed, the trace field of $K$ is $\Q(z)$, where $z$ satisfies $z^8 + z^6 - 3z^5 + 3z^4 + 5z^2 - z + 2$, and 
\[
\begin{aligned}
\mathcal{T}_K(t) &= \left(\frac{5}{4} z^{7} - \frac{9}{4} z^{6} + \frac{7}{2} z^{5} - \frac{9}{4} z^{4} + 7 z^{3} - 9 z^{2} + \frac{9}{4} z + \frac{23}{2}\right) t^{6} \\
&\quad + (-8 z^{7} + 20 z^{6} - 32 z^{5} + 16 z^{4} - 52 z^{3} + 84 z^{2} - 20 z - 88) t^{5} \\
&\quad + \left(\frac{37}{2} z^{7} - \frac{123}{2} z^{6} + 100 z^{5} - \frac{75}{2} z^{4} + 140 z^{3} - 271 z^{2} + \frac{123}{2} z + 259\right) t^{4}\\
&\quad + (-24 z^{7} + 88 z^{6} - 144 z^{5} + 48 z^{4} - 192 z^{3} + 392 z^{2} - 88 z - 368) t^{3}\\
&\quad + \left(\frac{37}{2} z^{7} - \frac{123}{2} z^{6} + 100 z^{5} - \frac{75}{2} z^{4} + 140 z^{3} - 271 z^{2} + \frac{123}{2} z + 259\right) t^{2}\\
&\quad + (-8 z^{7} + 20 z^{6} - 32 z^{5} + 16 z^{4} - 52 z^{3} + 84 z^{2} - 20 z - 88) t\\
&\quad + \frac{5}{4} z^{7} - \frac{9}{4} z^{6} + \frac{7}{2} z^{5} - \frac{9}{4} z^{4} + 7 z^{3} - 9 z^{2} + \frac{9}{4} z + \frac{23}{2}.
\end{aligned}
\]
Note that the coefficients of this polynomial are algebraic integers. We then check that $\mathcal{T}_K(-t^2)$ is irreducible over $\Q(z)$ using \cite{sagemath} so that $K$ cannot be freely 2-periodic. The knot $12n556$ has the same hyperbolic torsion, so that the same analysis applies.
\end{exmp}
\begin{rem}
	For knots that are freely $p$-periodic for $p$ odd, both discrete faithful representations are in the image of discrete faithful representations under $\widehat{p}$ as one can find with a similar calculation to that in Lemma \ref{prop:traceMinusTwo}. In short, one only has to pay attention to lifting issues when considering knots with even order free periods. These, however, are usually the most difficult case.
\end{rem}
\printbibliography
	
\end{document}